\documentclass[reqno,centertags,a4paper,12pt]{amsart}
\usepackage{amssymb}
\usepackage{amsthm}
\usepackage{eucal}
\usepackage{color, xcolor}
\usepackage[english]{babel}
\usepackage{graphicx}

\newcommand{\R}{\mathbb R}

\numberwithin{equation}{section}
\newtheorem{theorem}{Theorem}[section]
\newtheorem{proposition}[theorem]{Proposition}

\newtheorem{lemma}[theorem]{Lemma}

\newtheorem*{TA}{Theorem A}
\newtheorem*{TB}{Theorem B}

\begin{document}
\title[On  sharp global well-posedness]{On  sharp global well-posedness and Ill-posedness for a Fifth-order KdV-BBM type equation}

\author{X. Carvajal}
\address{Instituto de Matem\'atica, UFRJ, 21941-909, Rio de Janeiro, RJ, Brazil}
\email{carvajal@im.ufrj.br}
\author{M. Panthee}
\address{Department of Mathematics, IMECC-UNICAMP\\
13083-859, Campinas, S\~ao Paulo, SP,  Brazil}
\email{mpanthee@ime.unicamp.br}

\thanks{MP was partially supported by FAPESP (2016/25864-6) Brazil and CNPq (308131/2017-7) Brazil.}

\begin{abstract}
We consider the Cauchy problem associated to the recently derived higher order hamiltonian model for unidirectional water waves and prove global existence for given data in the  Sobolev space $H^s$, $s\geq 1$. We also prove an ill-posedness result by showing that the flow-map is not continuous if the given data has Sobolev regularity $s< 1$.  The results obtained in this work are sharp.
\end{abstract}

\maketitle

\section{Introduction}
In this work, we consider the Cauchy problem associated to the recently  introduced  higher order KdV-BBM type model by Bona et al. in  \cite{BCPS1} 
\begin{equation}\label{eq1.6}
\begin{cases}
\eta_t+\eta_x-\gamma_1\eta_{xxt}+\gamma_2\eta_{xxx}+\delta_1\eta_{xxxxt}+\delta_2\eta_{xxxxx}
+
\frac32\eta\eta_x
+ \gamma(\eta^2)_{xxx}
-\frac7{48}(\eta_x^2)_x-\frac18(\eta^3)_x=0,\\
\eta(x,0)=\eta_0(x),
\end{cases}
\end{equation}
where
\begin{equation}\label{gamas}
\begin{cases}
\gamma_1=\frac12(b+d-\rho),\\
\gamma_2=\frac12(a+c+\rho),\\
\delta_1=  \frac14\big[2(b_1+d_1)-(b-d+\rho)\big(\frac16-a-d\big)-d(c-a+\rho)\big],\\
\delta_2=  \frac14\big[2(a_1+c_1) -(c-a+\rho)\big(\frac16-a\big)+\frac1{3}\rho\big],\\
\gamma  =\frac1{24}\big[5-9(b+d)+9\rho\big].
\end{cases}
\end{equation}

The parameters appeared in \eqref{gamas} satisfy 
 $a+b+c+d=\frac13$, $\gamma_1+\gamma_2 =\frac16$, $\gamma=\frac1{24}(5-18\gamma_1)$, $\rho = b+d-\frac16$ and $\delta_2-\delta_1 = \frac{19}{360}-\frac16\gamma_1$ with $\delta_1>0$. 
 
The model in \eqref{eq1.6} describes the unidirectional propagation of water waves. The authors in \cite{BCPS1} used the second order approximation in the two-way model, the so-called $abcd$-system, introduced in \cite{BCS1, BCS2} and obtained a fifth order KdV-BBM type model \eqref{eq1.6}. 
 Also we note that,  the model \eqref{eq1.6} possesses an energy conservation law
\begin{equation}\label{apriori4}
 E(\eta(\cdot,t)):=  \frac12\int_{\mathbb{R}} \eta^2 + \gamma_1 (\eta_x)^2+\delta_1(\eta_{xx})^2\, dx= E(\eta_0),
\end{equation} 
when the parameter $\gamma = \frac7{48}$. In this particular case, the model \eqref{eq1.6} turns out to be hamiltonian. 

There are other higher order models of KdV and BBM type in the literature, see for example \cite{CL, DV, dullin, dullin04, olver1, olver} and references therein. These models are derived either by using Hamiltonian perturbation method \cite{olver1, olver} or by expanding Dirichlet--Neumann operator in the Zakharov--Craig--Sulem formulation \cite{DL}. Also, most of the higher order  KdV-BBM type models existing in the literature are either ill-posed or don't have hamiltonian structure, see for example \cite{ABN1, ABN2, AB} and references therein. The model \eqref{eq1.6} posed on half-line  is also studied in \cite{hongqiu}.

Well-posedness issues for the Cauchy problem \eqref{eq1.6} with initial data in the Sobolev spaces $H^s$ are studied in \cite{BCPS1}. More precisely, for given data  $\eta_0\in H^s (\R$), $s \geq 1$ the authors in \cite{BCPS1} proved the following  local well-posedness  result.

\begin{TA}\label{mainTh1}  Assume $\gamma_1, \delta_1 >0$.  
For any $s\geq 1$ and for given  $\eta_0\in H^s(\R)$, there exist a time $T_{\eta} =\dfrac{c_s}{\|\eta_0\|_{H^s}(1+\|\eta_0\|_{H^s})}$
 and a unique function  $\eta \in C([0,T_{\eta}];H^s)$ which 
 is a solution of the Cauchy problem \eqref{eq1.6}, posed with initial data $\eta_0$.  The solution $\eta$ 
varies continuously in $C([0,T_{\eta}];H^s)$ as $\eta_0$ varies in $H^s$.
 Moreover, for $R>0$, let $B_R$ denotes the ball of radius $R$  centered at the origin in $H^s(\mathbb{R})$ and let $T=T(R)>0$ denotes a uniform existence time for the Cauchy problem \eqref{eq1.6} with $\eta_0 \in B_R$. Then the correspondence $\eta_0 \mapsto \eta$ that associates to $\eta_0$ the solution $\eta$ of the Cauchy problem \eqref{eq1.6} is a real analytic mapping of $B_R$ to $C([0,T], H^s(\mathbb{R})\,)$.
\end{TA} 

In the case, when the parameter $\gamma=\frac{7}{48}$, conserved quantity \eqref{apriori4} allows one to get an {\em a priori} estimate in $H^2$ which in turn yields global well-posedness in $H^s$, for $s\geq 2$. For given data with certain range of Sobolev regularity below $H^2$, the authors in \cite{BCPS1} used splitting argument introduced in \cite{B1, B} (see also \cite{BS}) to extend the local solution to global in time. This argument is very powerful to get global solution of the Cauchy problem with low regularity data and is used by several authors, see for example \cite{BCPS1, BT, FLP1, FLP2} and references therein. It is worth noting that in \cite{B1, B} energy estimate was used to complete the iteration argument, while in \cite{FLP1} and \cite{FLP2} $L^pL^q$ estimates were used. As in the third order BBM equation \cite{BT}, the authors in  \cite{BCPS1} used energy estimate evolving high frequency part  (rough but small) of the initial data according to the original equation and the low frequency part (smooth but big) according to the difference equation to prove the global well-posedness of the Cauchy problem \eqref{eq1.6}  with given data in  range of Sobolev regularity $\frac32\leq s<2$. More specifically, the global well-posedness result proved in \cite{BCPS1} is the following.

\begin{TB}\label{mainTh3}  Assume $\gamma_1, \delta_1 > 0$.  
 Let $s\geq \frac32$ and $\gamma=\frac7{48}$.  Then  the solution to the Cauchy problem  \eqref{eq1.6}  given by Theorem \ref{mainTh1} can be extended to arbitrarily large time intervals $[0, T]$.  Hence the problem is globally well-posed in this case.
\end{TB}

Now, a natural question is: whether the results obtained in Theorems A and B are optimal? In this work we will try to respond this question.

Regarding global well-posedness,   in this work too we use the splitting argument as in \cite{BCPS1, BT} but perform the iteration argument other-way around, i.e., evolve the low frequency part according to the original equation and the high frequency part according to the difference equation so that the sum of two will give solution to the original problem. In this way, we get better decay (see \eqref{xeq12} and \eqref{xeq13} below) that allows to extend the local solution to the global one so as the achieve the full range of Sobolev regularity of local well-posedness, $s\geq 1$. This is the first main result of this work and is stated as follows.

\begin{theorem}\label{mainTh3.1}  Assume $\gamma_1, \delta_1 > 0$.  
 Let $1\leq s< 2$ and $\gamma=\frac7{48}$.  Then  for any given $T>0$, the solution to the Cauchy problem  \eqref{eq1.6}  given by Theorem A can be extended to the  time interval $[0, T]$.  Hence the Cauchy problem  \eqref{eq1.6} is globally well-posed in this case. In addition if $\eta_0 \in H^s$, one also has
\begin{equation}\label{Crecinorm1}
\eta(t) -S(t)\eta_0 \in H^2, \quad \textrm{for all time}\; t\in [0, T]
\end{equation}
and
\begin{equation}\label{Crecinorm2}
\sup_{t \in [0,T]}\|\eta(t) -S(t)\eta_0\|_{H^2} \lesssim (1+T)^{2-s},
\end{equation}
where $S(t)$ is as defined in \eqref{St} below.
\end{theorem}

The second main result of this work addresses the sharpness of the well-posedness issue by proving the following  ill-posedness result.
\begin{theorem}\label{mainTh1.1}  Assume $\gamma_1, \delta_1 >0$.  For any $s< 1$ the Cauchy problem \eqref{eq1.6} is ill-posed in $H^s(\mathbb{R})$ in the sense that there exists $T>0$ such that for any $t\in (0,T)$, the flow-map $\eta_0 \mapsto \eta(t)$ constructed in Theorem A is discontinuous at the origin from $H^s(\mathbb{R})$ endowed with the topology inducted by $H^s(\mathbb{R})$ into $\mathcal{D}'(\mathbb{R})$.
\end{theorem}

This negative result makes sense, because if one uses contraction mapping principle to prove local well-posedness, the flow-map turns out to be smooth.
In view of the results obtained in Theorems \ref{mainTh3.1} and \ref{mainTh1.1}, the global well-posedness of the Cauchy problem \eqref{eq1.6} for data $\eta_0\in H^s$, $s\geq 1$ is sharp.

Before leaving this section, we record some notations used in this work along with the structure. We use standard notations of the PDE and explain wherever necessary in their first appearance. The structure of the paper is as follows. In Section \ref{sec2} we prove the global well-posedness result stated in Theorem \ref{mainTh3.1} while Section \ref{sec-3} is devoted to prove the ill-posedness result stated in Theorem \ref{mainTh1.1}.

\section{Global Well-posedness  Results}\label{sec2}

In this section we will consider $\gamma = \frac7{48}$ and $1 \leq s <2$.   Let $T>0$ be large. Our objective is to extend the local solution to the Cauchy problem \eqref{eq1.6} given by Theorem A to a large time interval $[0,T]$, for any given $T>0$.  

We start by  writing  the Cauchy problem \eqref{eq1.6} in the following form
\begin{equation}\label{eq1.9}
\begin{cases}
i\eta_t = \phi(\partial_x)\eta + \tau (\partial_x)\eta^2 - \frac18\psi(\partial_x)\eta^3  -\frac7{48}\psi(\partial_x)\eta_x^2\, ,\\
 \eta(x,0) = \eta_0(x),
 \end{cases}
\end{equation}
where  $\phi(\partial_x)$, $\psi(\partial_x)$ and $\tau(\partial_x)$ are Fourier multiplier operators defined by, 
\begin{equation}\label{phi-D}
\widehat{\phi(\partial_x)f}(\xi):=\phi(\xi)\widehat{f}(\xi), \qquad \widehat{\psi(\partial_x)f}(\xi):=\psi(\xi)\widehat{f}(\xi) \;\;\; {\rm and }\;\;\; \widehat{\tau(\partial_x)f}(\xi):=\tau(\xi)\widehat{f}(\xi), 
\end{equation}
 with symbols
\begin{equation*}
  \phi(\xi)=\frac{\xi(1-\gamma_2\xi^2+\delta_2\xi^4)}{\varphi(\xi)}, \quad \psi(\xi)=\frac{\xi}{\varphi(\xi)} \quad  {\rm and} \quad \tau(\xi)=\frac{3\xi-4\gamma\xi^3}{4\varphi(\xi)}.
\end{equation*}

The common denominator
\begin{equation*}\label{def-phi}
 \varphi(\xi) := 1 + \gamma_1\xi^2+\delta_1\xi^4,
\end{equation*}
is strictly positive because the parameters  $\gamma_1$ and $  \delta_1$ are taken to be positive.

Consider first the linear Cauchy problem associated to \eqref{eq1.9}
\begin{equation}\label{eq1.10}
\begin{cases}
i\eta_t = \phi(\partial_x)\eta,\\
\eta(x,0) = \eta_0(x),
\end{cases}
\end{equation}
whose solution is given  by $\eta(t) = S(t)\eta_0$, where  $S(t)$ is defined via its Fourier transform \begin{equation}\label{St}
\widehat{S(t)\eta_0} = e^{-i\phi(\xi)t}\widehat{\eta_0}.
\end{equation}

Clearly, $S(t)$ is a unitary operator on $H^s$ for any $s \in \R$, so that
\begin{equation}\label{eq1.11}
\|S(t)\eta_0\|_{H^s} = \|\eta_0\|_{H^s},
\end{equation}
for all $t > 0$.
Duhamel's formula allows us to write the Cauchy problem  (\ref{eq1.9}) in the equivalent integral equation form,
\begin{equation}\label{eq1.12}
\eta(x,t) = S(t)\eta_0 -i\int_0^tS(t-t')\Big(\tau(\partial_x)\eta^2 - \frac18 \psi(\partial_x)\eta^3 -\frac7{48}\psi (\partial_x)\eta_x^2\Big)(x, t') dt'.
\end{equation}

Local well-posedness results for the Cauchy problem \eqref{eq1.6}  is obtained   in \cite{BCPS1} via the contraction mapping principle in the space $C([0,T];H^s)$, $s\geq 1$ using the Duhamel's formula (\ref{eq1.12}). To complete the contraction principle argument, the following estimates were crucial. In what follows we record these estimates from  \cite{BCPS1}  along with some  improvements, because they will be needed in our argument.
\begin{proposition}\label{BT1}
 For $s \ge 0$, there is a constant $C = C_s$ for which
\begin{equation}\label{bt}
\|\omega(\partial_x) (u v)\|_{H^s} \le C\|u\|_{H^s}\|v\|_{H^s}
\end{equation}
 where $\omega(\partial_x)$  is the Fourier multiplier operator 
with symbol
\begin{equation} \label{btx0}
\omega(\xi) \, = \, \frac{|\xi|}{1 + \xi^2}.
\end{equation} 
\end{proposition}
\begin{proof}
See Lemma 3.1  in \cite{BCPS1}.
\end{proof}

\begin{proposition}\label{P}
For any $s \ge 0$, there is a constant $C = C_s$ such that the inequality
\begin{equation}\label{bilin-1}
\|\tau(\partial_x) (\eta_1 \eta_2)\|_{H^s} \le C \| \eta_1\|_{H^s}  \| \eta_2\|_{H^s}
\end{equation}
and
\begin{equation}\label{xbilin-1}
\|\partial_x\tau(\partial_x) (\eta_1 \eta_2)\|_{H^1} \le C \| \eta_1\|_{H^1}  \| \eta_2\|_{H^1}
\end{equation}
holds, where the operator $\tau(\partial_x)$  is as defined in (\ref{phi-D}).
\end{proposition}
\begin{proof}
The proof of the inequality \eqref{bilin-1} is in Corollary 3.2 of \cite{BCPS1}. In order to prove \eqref{xbilin-1}, from definition of operator $\tau(\partial_x)$,  we have
\begin{equation}\label{1xeq13}
\begin{split}
\|\partial_x\tau(\partial_x) (\eta_1 \eta_2)\|_{H^1}=\|\langle \xi \rangle \,\xi \, \tau(\xi) \widehat{(\eta_1 \eta_2)}(\xi)\|_{L^2}
\end{split}
\end{equation}
and 
$$
|\xi \, \tau(\xi)|=\left|\dfrac{3\xi^2-4\gamma\xi^4}{4(1 + \gamma_1\xi^2+\delta_1\xi^4)}\right| \leq c.
$$
Thus,  since $H^1$ is an algebra, we get
\begin{equation*}
\begin{split}
\|\partial_x\tau(\partial_x) (\eta_1 \eta_2)\|_{H^1} \leq c\|\langle \xi \rangle  \widehat{(\eta_1 \eta_2)}(\xi)\|_{L^2}=c\|\eta_1 \eta_2\|_{H^1} \lesssim \|\eta_1 \|_{H^1}\|\eta_2\|_{H^1}.
\end{split}
\end{equation*}
\end{proof}

\begin{proposition}\label{P1}
For $s \ge \frac16$, there is a constant $C = C_s$ such that 
\begin{equation}\label{trilin-1}
\|\psi(\partial_x) (\eta_1 \eta_2 \eta_3)\|_{H^s} \le C \| \eta_1\|_{H^s}  \| \eta_2\|_{H^s}\| \eta_3\|_{H^s}
\end{equation}
and 
\begin{equation}\label{xtrilin-1}
\|\partial_x\psi(\partial_x) (\eta_1 \eta_2 \eta_3)\|_{H^1} \le C \| \eta_1\|_{H^1}  \| \eta_2\|_{H^1}\| \eta_3\|_{H^1}.
\end{equation}
\end{proposition}
\begin{proof}
The proof of the inequality \eqref{trilin-1} is in Proposition 3.3 of \cite{BCPS1}. In order to prove \eqref{xbilin-1}, from definition of operator $\psi(\partial_x)$,  we have
\begin{equation}\label{xtrilin-11}
\|\partial_x\psi(\partial_x) (\eta_1 \eta_2 \eta_3)\|_{H^1} =\|\langle \xi \rangle \,\xi \, \psi(\xi) \widehat{(\eta_1 \eta_2 \eta_3)}(\xi)\|_{L^2}
\end{equation}
and 
$$
|\xi \, \psi(\xi)|=\left|\dfrac{\xi^2}{1 + \gamma_1\xi^2+\delta_1\xi^4}\right| \leq c.
$$
Thus,  since $H^1$ is an algebra, we get
\begin{equation*}
\begin{split}
\|\partial_x\psi(\partial_x) (\eta_1 \eta_2 \eta_3)\|_{H^1} \leq c\|\langle \xi \rangle  \widehat{(\eta_1 \eta_2 \eta_3)}(\xi)\|_{L^2}=c\|\eta_1 \eta_2 \eta_3\|_{H^1} \lesssim \|\eta_1 \|_{H^1}\|\eta_2\|_{H^1}\|\eta_3\|_{H^1}.
\end{split}
\end{equation*}
\end{proof}

\begin{proposition}\label{P2}
For $s \ge 1$, the inequality 
\begin{equation}\label{Sharp1}
\|\psi(\partial_x)[ (\eta_1)_x(\eta_2)_x]\|_{H^s} \le C \| \eta_1\|_{H^s}\| \eta_2\|_{H^s}
\end{equation}
and
\begin{equation}\label{1Sharp1}
\|\partial_x \psi(\partial_x)[ (\eta_1)_x(\eta_2)_x]\|_{H^1} \le C \| \eta_1\|_{H^1}\| \eta_2\|_{H^1}
\end{equation}
hold.
\end{proposition}
\begin{proof}
The proof of the inequality \eqref{Sharp1} is in Lemma 3.5 of \cite{BCPS1}. In order to prove \eqref{1Sharp1}, from definition of operator $\psi(\partial_x)$,  we have
\begin{equation}\label{xtrilin-11}
\|\partial_x\psi(\partial_x)[ (\eta_1)_x(\eta_2)_x]\|_{H^1} =\|\langle \xi \rangle \,\xi \, \psi(\xi) \widehat{[ (\eta_1)_x(\eta_2)_x]}(\xi)\|_{L^2}.
\end{equation}
and 
$$
|\langle \xi \rangle \xi \, \psi(\xi)|=\dfrac{ \langle \xi \rangle \xi^2}{1 + \gamma_1\xi^2+\delta_1\xi^4} \leq c \dfrac{|\xi|}{1+\xi^2}=c \,\omega(\xi).
$$
Thus,  using Plancherel identity and Proposition \ref{BT1}, we get
\begin{equation*}
\begin{split}
\|\partial_x\psi(\partial_x) [ (\eta_1)_x(\eta_2)_x]\|_{H^1} \leq c\|\omega(\partial_x)[ (\eta_1)_x(\eta_2)_x]\|_{L^2} \lesssim \| (\eta_1)_x\|_{L^2} \| (\eta_2)_x\|_{L^2}\lesssim \|\eta_1 \|_{H^1}\|\eta_2\|_{H^1}.
\end{split}
\end{equation*}
\end{proof}

As mentioned in the introduction,  we will use the splitting argument introduced in \cite{B1, B} and way earlier in \cite{BS} to get global solution of the Cauchy problem \eqref{eq1.6}. 

Let $\eta_0 \in H^s$, $s\geq 1$,
we split the initial data $\eta_0=u_0+v_0$,  $\widehat{u_0}=\widehat{\eta_0} \chi_{\{|\xi| \leq N\}}$, where $N$ is a large number to be chosen later, it  is easy to see that $ u_0 \in H^\delta$ for any  $\delta \geq s$ and $v_0 \in H^s$. 
In fact we have
\begin{equation}\label{31}
\begin{split}
\|u_0\|_{L^2}&\leq \|\eta_0\|_{L^2},\\
\|u_0\|_{\dot{H}^{\delta}}&\leq \|\eta_0\|_{\dot{H}^s}\,N^{\delta-s},\qquad \delta\geq s,
\end{split}
\end{equation}
and
\begin{equation}\label{32}
\|v_0\|_{H^{\rho}} \leq \|\eta_0\|_{H^s}\,N^{(\rho-s)}, \qquad 0\leq\rho\leq s.
\end{equation}

For each parts  $u_0$ and $v_0$ of $\eta_0$ we associate the Cauchy problems
\begin{equation}\label{xeq1}
\begin{cases}
iu_t = \phi(\partial_x)u +F(u),\\
 u(x,0) = u_0(x),
 \end{cases}
\end{equation}
where $F(u)= \tau (\partial_x)u^2 - \frac18\psi(\partial_x)u^3  -\frac7{48}\psi(\partial_x)u_x^2\,$ and
\begin{equation}\label{xeq2}
\begin{cases}
iv_t = \phi(\partial_x)v + F(u+v)-F(\textcolor{red}{u}),\\
 v(x,0) = v_0(x),
 \end{cases}
\end{equation}
respectively, so that we have $\eta(x,t)=u(x,t)+v(x,t)$, solves the original Cauchy problem \eqref{eq1.6} in the common time interval of existence of $u$ and $v$. In what follows, we prove that there is a time $T_u$ such taht the Cauchy problem \eqref{xeq1} is locally well-posed in $[0, T_u]$. Fixing the solution $u$ of \eqref{xeq1}, we prove that there exists $T_v$ such that the Cauchy problem \eqref{xeq2} is locally well-posed in $[0, T_v]$. In this way, for $t_0\leq\min\{T_u, T_v\}$, $\eta=u+v$ solves the Cauchy problem \eqref{eq1.6} in the time interval $[0, t_0]$ for given data in $H^s$, $s\geq 1$. Our idea is to iterate this process maintaining $t_0$ as the length of existence time in each iteration  to cover any given time interval $[0, T]$.

By Theorem A the Cauchy problem  \eqref{xeq1} is locally well-posed in $H^s$, $s \geq 1$ with existence time given by $T_{u} =\dfrac{c_s}{\|u_0\|_{H^s}(1+\|u_0\|_{H^s})}$ and by  Theorem B globally well-posed in $H^s$, $s \geq 2$. Regarding the well-posedness of the Cauchy problem \eqref{xeq2} with variable coefficients that depend on $u$, we  have the following result.

\begin{theorem}\label{mainTh4}  Assume $\gamma_1, \delta_1 >0$ and $u$ the solution to the Cauchy problem \eqref{xeq1}.  
For any $s\geq 1$ and for given  $v_0\in H^s(\R)$, there exist a time $T_v =\dfrac{c_s}{(\|v_0\|_{H^s}+\|u_0\|_{H^s})(1+\|v_0\|_{H^s}+\|u_0\|_{H^s})}$
 and a unique function  $v \in C([0,T_v];H^s)$ which 
 is a solution of the IVP \eqref{xeq2}, posed with initial data $v_0$.  The solution $v$ 
varies continuously in $C([0,T_v];H^s)$ as $v_0$ varies in $H^s$.
\end{theorem} 
\begin{proof}
Using  Duhamel's formula, the equivalent integral equation to \eqref{xeq2} is
\begin{equation}\label{xeq6}
\begin{split}
v(x,t) &= S(t)v_0 -i\int_0^tS(t-t')\Big( F(u+v)-F(u)\Big)(x, t') dt'\\
&=: S(t)v_0 + h(x,t),
 \end{split}
\end{equation}
where 
\begin{equation}\label{xeq7}
 F(u+v)-F(u)= \tau (\partial_x)(v^2+2vu) - \frac18\psi(\partial_x)(3u^2v+3uv^2+v^3) -\frac7{48}\psi(\partial_x)(2u_xv_x+v_x^2).
\end{equation}

Let $u\in C([0,T_u];H^s)$ be the solution of Cauchy problem \eqref{xeq1}, given by Theorem A and satisfying 
\begin{equation}\label{eqofu}
\sup_{t \in [0, T_u]}\|u(t)\|_{H^s} \lesssim \|u_0\|_{H^s}.
\end{equation} 

Let
$$
X_T^a= \{ v \in C([0,T];H^s) :   \quad |||v|||:=\sup_{t \in [0, T]}\|v(t)\|_{H^s}  \leq a \}
$$
where $a:=2\|v_0\|_{H^s}$, and consider an application
$$
\Phi_u(v)(x,t) = S(t)v_0 -i\int_0^tS(t-t')\Big( F(u+v)-F(u)\Big)(x, t') dt'.
$$

We will prove that the application $\Phi_u(v)$ is a contraction on $X_T^a$.
By definition $S(t)$ is a unitary group in $H^s(\R)$. Then for $T\leq T_u$, we have
\begin{equation*}
\begin{split}
\|\Phi_u(v)\|_{H^s}& \leq \|v_0\|_{H^s} +T |||\tau (\partial_x)(v^2+2vu) - \frac18\psi(\partial_x)(3u^2v+3uv^2+v^3) \\
&\quad -\frac7{48}\psi(\partial_x)(2u_xv_x+v_x^2)|||.
\end{split}
\end{equation*}
The inequalities  \eqref{bilin-1}, \eqref{trilin-1}, \eqref{Sharp1} and \eqref{eqofu} yield
\begin{equation*}
\begin{split}
\|\Phi_u(v)\|_{H^s} &\leq \|v_0\|_{H^s} +T |||\tau (\partial_x)(v^2+2vu) - \frac18\psi(\partial_x)(3u^2v+3uv^2+v^3) \\
&\quad -\frac7{48}\psi(\partial_x)(2u_xv_x+v_x^2)||| \\
&\leq   \frac{a}2+cT|||v||| (|||v|||+ \|u_0\|_{H^s}) +cT|||v||| (\|u_0\|_{H^s}^2+\|u_0\|_{H^s} |||v|||+ |||v|||^2)
\\
&\leq   \frac{a}2+cT[a (a+ \|u_0\|_{H^s}) (1+a+ \|u_0\|_{H^s})].
\end{split}
\end{equation*}
If we choose 
$$
cT[(a+ \|u_0\|_{H^s}) (1+a+ \|u_0\|_{H^s})] = \frac12
$$
then  $\|\Phi_u(v)\|_{H^s} \leq a$, showing that $\Phi_u(v)$ maps
 the closed ball $X_T^a$ in $C([0,T];H^s)$ onto itself.   With the same choice of $a$ and $T$ and the same sort of 
estimates, one can prove that the application  $\Phi_u(v)$ is a contraction on $X_T^a$ with contraction constant equal to $\frac12$. The rest of
the proof is standard.
\end{proof}

In what follows, we record a lemma which will paly a fundamental role in the proof of the global well-posedness result.
\begin{lemma}\label{lemah1}
Let $u$ be the solution of the Cauchy problem \eqref{xeq1} and $v$ be the solution of the Cauchy problem \eqref{xeq2}, then $h=h(u,v)$ as defined in \eqref{xeq6} is in $C([0,t_0], H^2)$ and, 
\begin{equation}\label{estim1}
\|u(t_0)\|_{H^2} \lesssim N^{2-s} \quad \textrm{and}\quad \|h(t_0)\|_{H^2} \lesssim N^{s-3},
\end{equation} 
where $t_0 \sim N^{-2(2-s)}$.
\end{lemma}

\begin{proof}
Observe that the energy conservation law \eqref{apriori4}, gives
$$
\|u(t_0)\|_{H^2} \sim \sqrt{E(u(t_0))}=\sqrt{E(u_0)}\sim \|u_0\|_{H^2} \lesssim N^{2-s}.
$$
On the other hand, from \eqref{xeq6} and \eqref{xeq7},  we have  for $ 1 \leq \delta \leq s $
\begin{equation}
\begin{split}
\|h(t_0)\|_{H^{\delta}} & = \Big \|\int_0^{t_0} S(-t')\Big( F(u+v)-F(u)\Big)(x, t') dt'\Big\|_{H^{\delta}} \\
&\leq \int_0^{t_0}\|S(-t')\Big( F(u+v)-F(u)\Big)(x, t') dt'\|_{H^{\delta}} \\
&\leq \int_0^{t_0}(\|\tau (\partial_x)(v^2+2vu) \|_{H^{\delta}}+\frac18\|\psi(\partial_x)(3u^2v+3uv^2+v^3) \|_{H^{\delta}}\\
&\quad +\frac7{48}\|\psi(\partial_x)(2u_xv_x+v_x^2)\|_{H^{\delta}}) dt'.
\end{split}
\end{equation}

Now, using the Propositions \ref{P}, \ref{P1} and \ref{P2}, we arrive to
\begin{equation}
\begin{split}
\|h(t_0)\|_{H^{\delta}} 
\lesssim &\int_0^{t_0}(\|v\|_{H^{\delta}}^2+\|v\|_{H^{\delta}}\|u\|_{H^{\delta}}+\|u\|_{H^{\delta}}^2\|v\|_{H^{\delta}}+\|u\|_{H^{\delta}}\|v\|_{H^{\delta}}^2+\|v\|_{H^{\delta}}^3 )dt'.
\end{split}
\end{equation}

The local theory and the inequalities \eqref{31} and \eqref{32} imply $\|v\|_{H^{\delta}} \lesssim N^{\delta -s}$ and $\|u\|_{H^{\delta}} \lesssim c$.

 Thus, if $\delta =1$ and $s \geq1$, we have
\begin{equation}\label{eq17}
\begin{split}
\|h(t_0)\|_{H^{1}} 
&\lesssim \int_0^{t_0}(N^{2(1 -s)}+N^{(1 -s)}+N^{3(1 -s)})dt' \\
&\lesssim  t_0(N^{2(1 -s)}+N^{(1 -s)}+N^{3(1 -s)})\\
&\lesssim  N^{-2(2-s)} (N^{2(1 -s)}+N^{(1 -s)}+N^{3(1 -s)})\\
&\lesssim N^{s-3}+N^{-2}+N^{-s-1}\\
&\lesssim N^{s-3}.
\end{split}
\end{equation}

Furthermore
\begin{equation}
\begin{split}
\|\partial_x h(t_0)\|_{H^{1}} \leq &\int_0^{t_0}(\|\partial_x \tau (\partial_x)(v^2+2vu) \|_{H^{1}}+\frac18\|\partial_x\psi(\partial_x)(3u^2v+3uv^2+v^3) \|_{H^{1}}\\
&+\frac7{48}\|\partial_x\psi(\partial_x)(2u_xv_x+v_x^2)\|_{H^{1}}) dt'.
\end{split}
\end{equation}

Using the Propositions \ref{P}, \ref{P1} and \ref{P2}, we obtain
\begin{equation}
\begin{split}
\|\partial_x h(t_0)\|_{H^{1}} \lesssim &\int_0^{t_0}(\|v\|_{H^{1}}^2+\|v\|_{H^{1}}\|u\|_{H^{1}} +\|u\|_{H^{1}}^2\|v\|_{H^{1}}+\|v\|_{H^{1}}^2\|u\|_{H^{1}}+\|v\|_{H^{1}}^3) dt'.
\end{split}
\end{equation}

Similarly, as in \eqref{eq17} one can prove
\begin{equation}\label{eq18}
\begin{split}
\|\partial_x h(t_0)\|_{H^{1}} 
\lesssim &N^{s-3}.
\end{split}
\end{equation}

Combining \eqref{eq17} and \eqref{eq18}, one gets
\begin{equation}
\begin{split}
\|h(t_0)\|_{H^{2}} \sim \|h(t_0)\|_{H^{1}}+ \|\partial_xh(t_0)\|_{H^{1}} \lesssim N^{s-3},
\end{split}
\end{equation}
which completes the proof of lemma.
\end{proof}

Now we are in position to supply the proof of the first main result of this work.

\begin{proof}[Proof of Theorem \ref{mainTh3.1}]  Let $\eta_0\in H^s$, $1\leq s< 2$ and $T>0$ be any given number. As discussed above, we split the initial data $\eta_0=u_0+v_0$ so that $u_0$ and $v_0$ satisfy the growth conditions \eqref{31} and \eqref{32} respectively.

We evolve $u_0$ according to the Cauchy problem \eqref{xeq1} and $v_0$ according to the Cauchy problem \eqref{xeq2}. Using Theorems A and \ref{mainTh4} we respectively obtain solutions $u$ and $v$ so that the sum $\eta=u+v$ solves the Cauchy problem \eqref{eq1.6} in the common time interval of existence of $u$ and $v$.

Observe that from \eqref{apriori4} and  \eqref{31},  we have
\begin{equation}\label{xeq4}
E(u(t))=E(u_0) \sim\|u_0\|_{H^2}^2 \lesssim N^{2(2-s)}
\end{equation}
and the  local existence time in $H^2$, given in Theorem A is estimated by
\begin{equation}\label{xeq5}
\begin{split}
T_u&=\dfrac{c_s}{\|u_0\|_{H^2}(1+\|u_0\|_{H^2})}\\
 &\geq \dfrac{c_s}{ N^{(2-s)}(1+ N^{(2-s)})} \\
&\geq \dfrac{c_s}{ N^{2(2-s)}}=:t_0.
 \end{split}
\end{equation}

We observe that $(\|v_0\|_{H^s}+\|u_0\|_{H^s})(1+\|v_0\|_{H^s}+\|u_0\|_{H^s})\lesssim \|\eta_0\|_{H^s}(1+\|\eta_0\|_{H^s})=C_s$, therefore
\begin{equation}\label{uxeq5}
T_v =\dfrac{c_s}{(\|v_0\|_{H^s}+\|u_0\|_{H^s})(1+\|v_0\|_{H^s}+\|u_0\|_{H^s})} \geq \dfrac{c_s}{C_s} \geq t_0.
\end{equation}

The inequalities \eqref{xeq5} and \eqref{uxeq5} imply that the solutions $u$ and $v$ are both defined in the same time interval $[0,t_0]$.

The inequality \eqref{xeq4} implies that
\begin{equation}\label{0xeq4}
 t_0 \lesssim \dfrac{1}{E(u_0)}.
\end{equation}

In view of \eqref{xeq6} the local solution $v \in H^s$ is given by 
\begin{equation}\label{xeq6-1}
v(x,t) = S(t)v_0 + h(x,t).
\end{equation}
Therefore, in the time $t_0 \sim  N^{-2(2-s)}$, the solution $\eta$ can be written as
\begin{equation}\label{00xeq9}
\eta(t)=u(t)+ v(t) = u(t)+ S(t)v_0 +h(t), \quad t\in [0,t_0].
\end{equation}

At the time $t=t_0$, we have
\begin{equation}\label{xeq9}
\eta(t_0)=u(t_0)+ S(t_0)v_0 +h(t_0)=:u_1+v_1,
\end{equation}
where 
\begin{equation}\label{decomp1}
u_1=u(t_0) +h(t_0) \quad \textrm{and} \quad v_1 =S(t_0)v_0.
\end{equation} 
In the time $t_0$  we consider the new initial data $u_1$, $v_1$ and evolve them according to the Cauchy problems \eqref{xeq1} and \eqref{xeq2} respectively, and continue iterating this process. In each iteration we consider the decomposition of the initial data as in \eqref{decomp1}. Therefore $v_1, \dots, v_k = S(k t_0)v_0$ have the same $H^s$-norm of $v_0$ i.e. $\|v_k\|_{H^s} = \|v_0\|_{H^s}$. We expect that $u_1, \dots, u_k$ also have the same properties of $u_0$, i.e., the same growth properties as that of $u_0$, in order to ensure the same existence time interval $[0,t_0]$ in each iteration and glue them to cover the whole time interval $[0,T]$, then extending the solution of the systems \eqref{xeq1} and \eqref{xeq2}. This fact is proved by induction. Here we will prove only the case $k = 1$ and note that a similar argument works in the general case. In order to attain this goal we will use the energy conservation \eqref{apriori4}.

We have 
\begin{equation}\label{xeq10}
\begin{split}
E(u_1)&=E(u(t_0) +h(t_0))=E(u(t_0))+ \big[E(u(t_0) +h(t_0))-E(u(t_0))\big]\\
 &=:E(u(t_0))+\mathcal{X}.
\end{split}
\end{equation}

Now,
\begin{equation}\label{xeq11}
\begin{split}
\mathcal{X}&=2 \int_{\R}u(t_0)h(t_0) dx+\int_{\R}h(t_0)^2dx+2\gamma_1 \int_{\R}u_x(t_0)h_x(t_0) dx\\
 & \quad+\gamma_1\int_{\R}h_x(t_0)^2dx+2 \delta_1\int_{\R}u_{xx}(t_0)h_{xx}(t_0) dx+\delta_1\int_{\R}h_{xx}(t_0)^2dx\\
&\leq  2\|u(t_0)\|_{L^2}\|h(t_0)\|_{L^2}+\|h(t_0)\|_{L^2}^2+ \gamma_1 (2 \|u_x(t_0)\|_{L^2}\|h_x(t_0)\|_{L^2}+\|h_x(t_0)\|_{L^2}^2)\\
&\quad+\delta_1(2\|u_{xx}(t_0)\|_{L^2}\|h_{xx}(t_0)\|_{L^2}+\|h_{xx}(t_0)\|_{L^2}^2).
\end{split}
\end{equation}

Using the Lemma \ref{lemah1}, the estimate \eqref{xeq11} yields
\begin{equation}\label{xeq12}
\begin{split}
\mathcal{X} &\lesssim   N^{2-s}N^{s-3}+N^{2(s-3)}+ (\gamma_1+\delta_1) (N^{2-s}N^{s-3}+N^{2(s-3)})\\
& \lesssim   N^{-1}.
\end{split}
\end{equation}
Combining \eqref{xeq10}, \eqref{xeq11} and \eqref{xeq12}, we conclude that 
\begin{equation}\label{xeq13}
\begin{split}
E(u_1) & \leq E(u(t_0))+cN^{-1}.
\end{split}
\end{equation}
The number of steps in the iteration to cover the given time interval $[0, T]$ is
$$\dfrac{T }{t_0}\sim TN^{2(2-s)}.$$
Thus, by \eqref{xeq13}, for this to happen we need that
$$
TN^{2(2-s)} N^{-1}  \lesssim N^{2(2-s)},
$$
which is possible if $1\leq s < 2$ and
$N = N(T) = T$.

From  the discussion  above we see that in each iteration  one has
$$
\|u_k\|_{H^2}^2 \sim E(u_k) \lesssim N^{2(2-s)},\;{\textrm{uniformly and}}\quad \|v_k\|_{H^2}=\|v_0\|_{H^2}.
$$

 Finally, let $t \in [0,T]$, then there exist $k \geq0 $ an  integer, such that $t=k t_0 +\tau$, for some $\tau \in [0,t_0]$. In the  $k^{{\textrm{th}}}$-iteration (see equality \eqref{00xeq9}), one gets
\begin{equation}\label{xeq13-1}
\begin{split}
\eta(t) &=u(\tau)+ S(\tau)v_k +h(\tau) \\
&=u(\tau)+ S(\tau)S(kt_0)v_0 +h(\tau) \\
 &= S(t)\eta_0+ u(\tau)- S(t)u_0+h(\tau).
\end{split}
\end{equation}

Thus
\begin{equation}\label{xeq13-2}
\begin{split}
\eta(t) - S(t)\eta_0= u(\tau)- S(t)u_0+h(\tau),
\end{split}
\end{equation}
and this completes  the proof of Theorem \ref{mainTh3.1}. 
 
\end{proof}

\section{Ill-posedness Result}\label{sec-3}
In this section, we  consider the ill-posedness issue for the Cauchy problem \eqref{eq1.6} and prove the result stated in Theorem \ref{mainTh1.1}. The ill-posedness result stated in Theorem \ref{mainTh1.1} makes sense, because if one uses contraction mapping principle to prove local well-posedness, the flow-map turns out to be smooth.  The main idea in the proof  is to find an appropriate initial data $\eta_N$  and follow the argument introduced in \cite{BejT, MV, MP-1} to show that the flow-map cannot be continuous.

\begin{proof}[Proof of Theorem \ref{mainTh1.1}] 
Let $N\gg 1$, $\alpha$ a constant to be chosen later, $I_N:= [N, N+2\alpha]$ and  define an initial data via the Fourier transform
\begin{equation}\label{eq4.2}
\widehat{\eta_N}(\xi):= N^{-1}\alpha^{-\frac12}\big[\chi_{\{I_N\}}(\xi) + \chi_{\{-I_N\}}(\xi)\big].
\end{equation}
A simple calculation shows that $\|\eta_N\|_{H^1}\sim 1$ and $\lim_{N\to \infty}\|\eta_N\|_{H^s}=0$ if $s<1$.

First, we use this example to show the failure of continuity of the second iteration of the Picard scheme
\begin{equation}\label{eq4.04}
I_2(h,h,x,t):=\int_0^tS(t-t')\Big(\tau(\partial_x)\big[S(t')h\big]^2
 -\frac7{48} \psi(\partial_x)\partial_x\big[S(t')h\big]^2\Big)  dt'
\end{equation}
at the origin from $H^s(\R)$ to even $\mathcal{D}'(\R)$ for $s<1$.  

In what follows, considering $h=\eta_N$, we calculate the $H^s$ norm of $I_2(\eta_N,\eta_N,x,t)=:I_2$, $0< t < T$.
Taking the Fourier transform in the space variable $x$, we get
\begin{equation}\label{eq4.05}
\begin{split}
\mathcal{F}_x  (I_2\,)(\xi) &=\int_0^t e^{-i(t-t')\phi(\xi)}\left(\tau(\xi)\widehat{\big(S(t')\eta_N\big)^2}(\xi)-\frac7{24}\psi(\xi)\widehat{\big(S(t')\partial_x\eta_N\big)^2}(\xi)\right)  dt'\\
&= \int_0^t e^{-i(t-t')\phi(\xi)}\left(\frac{3\xi-4\gamma\xi^3}{4\varphi(\xi)}\widehat{\big(S(t')\eta_N\big)^2}(\xi)-\frac7{48}\frac{\xi}{\varphi(\xi)}\widehat{\big(S(t')\partial_x\eta_N\big)^2}(\xi)\right)  dt'\\
&= \int_0^t e^{-i(t-t')\phi(\xi)}\left(\frac{3\xi-4\gamma\xi^3}{4\varphi(\xi)}\int_{\R}e^{-it'\phi(\xi-\xi_1)}\widehat{\eta_N}(\xi-\xi_1)e^{-it'\phi(\xi_1)}\widehat{\eta_N}(\xi_1) d\xi_1\right.\\
&\qquad\qquad\qquad  \left.-\frac7{48}\frac{\xi}{\varphi(\xi)}\int_{\R}e^{-it'\phi(\xi-\xi_1)}(\xi-\xi_1)\widehat{\eta_N}(\xi-\xi_1)e^{-it'\phi(\xi_1)}\xi_1\widehat{\eta_N}(\xi_1) d\xi_1\right)  dt'\\
&= \int_{\R} e^{-it\phi(\xi)}\left(\frac{3\xi-4\gamma\xi^3}{4\varphi(\xi)}-\frac7{48}\frac{\xi\xi_1(\xi-\xi_1)}{\varphi(\xi)}\right)\widehat{\eta_N}(\xi_1)\widehat{\eta_N}(\xi-\xi_1)   \int_0^te^{it'[\phi(\xi)-\phi(\xi-\xi_1)-\phi(\xi_1)]} dt'd\xi_1.
\end{split}
\end{equation}

We have that
\begin{equation}\label{eq4.06}
 \int_0^te^{it[\phi(\xi)-\phi(\xi-\xi_1)-\phi(\xi_1)]} dt'd\xi_1 = 
\frac{e^{it[\phi(\xi)-\phi(\xi-\xi_1)-\phi(\xi_1)]}-1}{i[\phi(\xi)-\phi(\xi-\xi_1)-\phi(\xi_1)]}.
\end{equation}

Now, inserting \eqref{eq4.06} in \eqref{eq4.05}, one obtains
\begin{equation}\label{eq4.07}
\mathcal{F}_x (I_2)(\xi)
=-i\int_{\R}\frac{\xi}{4\varphi(\xi)}\left(3-4\gamma\xi^2-\frac7{12}\xi_1(\xi-\xi_1)\right)\widehat{\eta_N}(\xi-\xi_1) \widehat{\eta_N}(\xi_1)
 e^{-it\phi(\xi)}\frac{e^{it\Theta(\xi, \xi_1)}-1}{\Theta(\xi, \xi_1)}d\xi_1,
\end{equation}
where $\Theta(\xi, \xi_1):=\phi(\xi)-\phi(\xi-\xi_1)-\phi(\xi_1)$.

Let us define a set $K$ by
\begin{equation}\label{def-K}
K:= \{ \xi_1: \;\xi-\xi_1 \in I_N,\;\, \xi_1\in -I_N\}\cup \{ \xi_1:\; \xi_1 \in I_N, \;\,\xi-\xi_1\in -I_N\}.
\end{equation}
With this notation, one infers for $|\xi| < \alpha$, that 
\begin{equation}\label{eq4.071}
\mathcal{F}_x (I_2)(\xi)
=-iN^{-2}\alpha^{-1} e^{-it\phi(\xi)}\int_{K}\frac{\xi}{4\varphi(\xi)}\left(3-4\gamma\xi^2-\frac7{12}\xi_1(\xi-\xi_1)\right)
\frac{e^{it\Theta(\xi, \xi_1)}-1}{\Theta(\xi, \xi_1)}d\xi_1.
\end{equation}

Therefore,
\begin{equation}\label{eqil1}
\| I_2\|_{H^s}^2 \gtrsim \int_{-\alpha/2}^{\alpha/2}\langle \xi\rangle^{2s} \alpha^{-2}N^{-4} 
\left| \int_{K}\frac{\xi g(\xi, \xi_1)}{4\varphi(\xi)} \frac{e^{it\Theta(\xi, \xi_1)}-1}{\Theta(\xi, \xi_1)}d\xi_1\right|^2 d\xi,
\end{equation}
where $g(\xi, \xi_1):=3-4\gamma\xi^2-\frac7{12}\xi_1(\xi-\xi_1)$.

We have that $|K|\gtrsim \alpha$. Now, we move to show that the magnitude of $\Theta(\xi, \xi_1)$ in the set $K$ is very very small.

Note that, the phase function $\phi(\xi)$ is odd. Consider the parameters $\gamma_1, \gamma_2, \delta_1$ and $\delta_2$ all positive. With these considerations, one can write
\begin{equation}\label{Theta-1}
\begin{split}
\Theta(\xi, \xi_1) &= \int_0^1\frac{d}{dt}[\phi(t\xi)-\phi(t\xi-\xi_1)]dt\\
& =\int_0^1[\xi\phi'(t\xi)-\xi\phi'(t\xi-\xi_1)]dt.
\end{split}
\end{equation}

Using triangle inequality, one easily obtains from \eqref{Theta-1} that
\begin{equation}\label{Theta-2}
|\Theta(\xi, \xi_1)| \leq |\xi|\int_0^1|\phi'(t\xi)|dt+ |\xi|\int_0^1|\phi'(t\xi-\xi_1)|dt.
\end{equation}

Let \begin{equation}\label{p-1}
p(\xi):=\frac{1-\gamma_2\xi^2+\delta_2\xi^4}{1+\gamma_1\xi^2+\delta_1\xi^4},
\end{equation}
 so that, one has $\phi(\xi) =\xi p(\xi)$ and $\phi'(\xi) = p(\xi)+\xi p'(\xi)$. Observe that
\begin{equation}\label{p-2}
p'(\xi)=\frac{-2(\gamma_1+\gamma_2)\xi+4(\delta_2-\delta_1)\xi^3+2(\gamma_2\delta_1+\gamma_1\delta_2)\xi^5}{(1+\gamma_1\xi^2+\delta_1\xi^4)^2}.
\end{equation}

For any $x\in \R$, one can infer $|\phi'(x)|\leq c$. In the domain of integration in the RHS of \eqref{eqil1}, we have $|\xi| \leq \alpha/2$ . Therefore, in the light of the definition of $p$ in \eqref{p-1} and the expression for $p'$ in \eqref{p-2}, one can easily obtain from \eqref{Theta-2} that
\begin{equation}\label{Theta-3}
|\Theta(\xi, \xi_1)| \leq C\alpha.
\end{equation}
Hence, considering $ \alpha = \dfrac{\pi}{4Ct}$ in \eqref{Theta-3}, for any fixed $t$, we can obtain
\begin{equation}\label{lb-t}
\left|  \frac{e^{it\Theta(\xi, \xi_1)}-1}{\Theta(\xi, \xi_1)}\right|\geq \dfrac{\sin(t\Theta(\xi, \xi_1) \,)}{t\Theta(\xi, \xi_1)}t\geq \cos(t\Theta(\xi, \xi_1)\,) t  \geq t\sqrt{2}/2.
\end{equation}

Now, using mean value theorem for the integrals and the lower bound \eqref{lb-t}, one can infer that
\begin{equation}\label{lb-t2}
\left| \int_{K} \frac{e^{it\Theta(\xi, \xi_1)}-1}{\Theta(\xi, \xi_1)}d\xi_1\right|\geq t|K|\sqrt{2}/2.
\end{equation}

In the set $K$, we have that $|\frac{\xi g(\xi, \xi_1)}{4\varphi(\xi)}|\sim \alpha N^2$. Using this last information along with the lower bound \eqref{lb-t2} in \eqref{eqil1}, after easy calculations one can obtain
\begin{equation}\label{eqil2}
 \| I_2(\eta_N,\eta_N,x,t)\|_{H^s}^2\gtrsim  N^{-4}N^4|K|^2\alpha t^2\gtrsim \alpha^{3} t^2\sim \dfrac{1}{t}.
\end{equation}

By construction, $\|\eta_N\|_{H^s}\to 0$ for any $s<1$, therefore, taking $\eta_0=\eta_N$,  \eqref{eqil2} ensures that for any fixed $t>0$, the application $\eta_0\mapsto I_2(\eta_0, \eta_0, t)$ is not continuous at the origin from $H^s(\R)$ to even $\mathcal{D}'(\R)$.

Now, our idea is to prove that the discontinuity of $\eta_0\mapsto I_2(\eta_0, \eta_0, t)$ at the origin implies the discontinuity of the flow-map $\eta_0\mapsto \eta(t)$.

Recall that the application that takes initial data to the solution constructed in Theorem A  is real analytic. Using analyticity of the flow-map,  there exist $T>0$ and $\epsilon_0 >0$ such that for any $|\epsilon|\leq \epsilon_0$, any $\|h\|_{H^1(\R)}\leq 1$ and $0\leq t\leq T$, one has
\begin{equation}\label{eq2.7}
\eta(\epsilon h, t) = \epsilon S(t)h + \sum_{k=2}^{+\infty} \epsilon^k I_k(h^k, t),
\end{equation}
where $h^k:= (h, h, \cdots, h)$, $h^k\mapsto I_k(h^k, t)$ is a $k$-linear continuous map from $H^1(\R)^k$ into $C([0, T]; H^1(\R))$ and the series converges absolutely in $C([0, T]; H^1(\R))$.

From \eqref{eq2.7}, we have that
\begin{equation}\label{eq2.8}
\eta(\epsilon \eta_N, t) -\epsilon^2I_2(\eta_N, \eta_N, t) = \epsilon S(t)\eta_N + \sum_{k=3}^{+\infty} \epsilon^k I_k(\eta_N^k, t).
\end{equation}

Also, we have that
\begin{equation}\label{eq2.9}
\|S(t)\eta_N\|_{H^s}\leq \|\eta_N\|_{H^s} \sim N^{s-1}
\end{equation}
and
\begin{equation}\label{eq2.10}
\Big\|\sum_{k=3}^{+\infty} \epsilon^k I_k(\eta_N^k, t)\Big\|_{H^1}\leq \Big(\frac{\epsilon}{\epsilon_0}\Big)^3 \sum_{k=3}^{+\infty} \epsilon_0^k \|I_k(\eta_N^k, t\|_{H^1}
\leq C\epsilon^3.
\end{equation}

Therefore, from \eqref{eq2.8} in view of \eqref{eq2.9} and \eqref{eq2.10} we get, for any $s<1$,
\begin{equation}\label{eq2.11}
\sup_{t\in [0, T]}\|\eta(\epsilon\eta_N, t) -\epsilon^2 I_2(\eta_N, \eta_N, t)\|_{H^s}\leq O(N^s) +C\epsilon^3.
\end{equation}

Now, if we fix $0<t<1$, take $\epsilon$ small enough and then $N$ large enough, and take an account of  \eqref{eqil2}; the estimate \eqref{eq2.11} yields that $\epsilon^2 I_2(\eta_N, \eta_N, t)$ is a good approximation of $\eta(\epsilon\eta_N, t)$ in $H^s(\R)$ for any $s<1$.

If we choose $\epsilon \ll 1$, from \eqref{eq2.8}, \eqref{eq2.9} and \eqref{eq2.10},  for a fixed $t>0$, we get 
\begin{equation}\label{eq2.12}
\begin{split}
\|\eta(\epsilon\eta_N, t)\|_{H^s}&\geq \epsilon^2\|I_2(\eta_N, \eta_N, t\|_{H^s} -\epsilon\|S(t)\eta_N\|_{H^s}-\sum_{k=3}^{+\infty} \epsilon^k \|I_k(\eta_N^k, t\|_{H^s}\\
&\geq C_0\epsilon^2 -C_1\epsilon^3 - C\epsilon N^{s-1}\\
&\geq \frac{C_0}{2}\epsilon^2 -C\epsilon N^{s-1}.
\end{split}
\end{equation}

If we fix the $0<\epsilon \ll 1$ chosen earlier and choose $N$ large enough, then for any $s<1$, the estimate \eqref{eq2.12} yields,
\begin{equation}\label{eq2.13}
\|\eta(\epsilon\eta_N, t)\|_{H^s}\geq \frac{C_0}{4} \epsilon^2.
\end{equation}

Note that, $\eta(0, t) \equiv 0$ and $\|\eta_N\|_{H^s}\to 0$ for any $s<1$. Therefore, taking $N\to \infty$ we conclude that the flow-map $\eta_0\mapsto \eta(t)$ is discontinuous at the origin from $H^s(\R)$ to $C([0, 1]; H^s(\R)$, for $s<1$. Moreover, as $\eta_N \rightharpoonup 0$ in $H^1(\R)$, we also have that the flow-map is discontinuous from $H^1(\R)$ equipped with its weak topology inducted by  $H^s(\R)$ with values even in $\mathcal{D}'(\R)$.
\end{proof}

\bigskip
{\bf Acknowledgment.}    The authors would like to thank the unanimous referee whose comments helped to improve the original manuscript.


\end{document}